\documentclass[smallextended]{svjour3}

\usepackage{amssymb,amsmath,array}
\usepackage{color}
\usepackage{verbatim}
\usepackage{amsfonts}
\usepackage{epsfig}
\usepackage{multirow}

\spnewtheorem{assumption}{Assumption}{\bf}{\it}
\spnewtheorem{corrolary}{Corrollary}{\bf}{\it}

\providecommand{\rset}[1]{\mathbb{R}^}
\providecommand{\abs}[1]{\lvert#1\rvert}
\providecommand{\norm}[1]{\lVert#1\rVert}

\smartqed

\begin{document}

\title{Efficient random coordinate descent algorithms for large-scale
structured nonconvex optimization
\thanks{The research leading to these results has
received funding from: the European Union (FP7/2007--2013) EMBOCON
under grant agreement no 248940; CNCS (project TE-231,
19/11.08.2010); ANCS (project PN II, 80EU/2010); POSDRU/89/1.5/S/62557. \\
 }}
\titlerunning{Random coordinate descent algorithms for nonconvex
optimization}

\author{Andrei Patrascu \and Ion Necoara}
\authorrunning{A. Patrascu \and I. Necoara}

\institute{A. Patrascu and I. Necoara  are  with the Automatic
Control and Systems Engineering Department, University Politehnica
Bucharest, 060042 Bucharest, Romania, Tel.: +40-21-4029195,  Fax:
+40-21-4029195; \email{\texttt{\{ion.necoara,
andrei.patrascu\}@acse.pub.ro}}}

\date{Revised: January 2014}

\maketitle

\begin{abstract}
In this paper we analyze several new methods for solving nonconvex
optimization problems with the objective function consisting of a sum of
two  terms: one is nonconvex and smooth, and another is convex but
simple and its structure is known. Further, we consider both cases:
unconstrained and linearly constrained nonconvex problems. For
optimization problems of the above structure,  we propose random
coordinate descent algorithms and analyze  their convergence
properties. For the general case, when the objective function is
nonconvex and composite we prove asymptotic convergence for the
sequences generated by our algorithms to stationary points  and
sublinear rate of convergence in expectation for some optimality
measure. We also present
extensive numerical experiments for evaluating the performance of
our algorithms in comparison with state-of-the-art methods.
\end{abstract}

\section{Introduction}

Coordinate descent methods are among the first algorithms used for
solving general minimization problems  and are some of the most
successful in the large-scale optimization field \cite{Ber:99}.
Roughly speaking, coordinate descent methods  are based on the
strategy of updating one (block) coordinate of the vector of
variables per iteration using some index selection procedure (e.g.
cyclic, greedy, random). This often reduces drastically the
complexity per iteration and memory requirements, making these
methods simple and scalable. There exist numerous papers dealing
with the convergence analysis of this type of methods
\cite{Aus:76,LinLuc:09,Nec:13,NecPat:12,Pow:73,TseYun:09}, which
confirm the difficulties encountered in proving the convergence for
nonconvex and nonsmooth objective functions. For instance, regarding
coordinate minimization of nonconvex functions, Powell \cite{Pow:73}
provided some examples of differentiable functions whose properties
lead the algorithm to a closed loop. Also,  proving  convergence of
coordinate descent methods for minimization of nondifferentiable
objective functions is challenging \cite{Aus:76,FerRic:13}. However,
for nonconvex and nonsmooth objective functions with certain
structure (e.g. composite objective functions) there are available
convergence results for coordinate descent methods based on greedy
index selection \cite{Bec:12,LinLuc:09,TseYun:09} or random index
selection \cite{LuXia:13}. Recently, Nesterov \cite{Nes:10} derived
complexity results for random coordinate gradient descent methods
for  solving smooth and convex optimization problems. In
\cite{RicTac:11} the authors generalized Nesterov's results to
convex problems with composite objective functions. An extensive
complexity analysis of coordinate gradient descent methods for
solving linearly constrained optimization problems with convex
(composite) objective function can be found
in~\cite{Bec:12,Nec:13,NecPat:12,NecNes:12}.

In this paper we also consider large-scale nonconvex optimization
problems with the objective function consisting of  a sum of two
terms: one is nonconvex, smooth and given by a black-box oracle, and
another is convex but simple and its structure is known. Further, we
analyze unconstrained   but also  singly linearly  constrained
nonconvex problems. We also assume  that the dimension of the
problem is so large that traditional optimization methods cannot be
directly employed since basic operations, such as the updating of
the gradient, are too computationally expensive. These types of
problems arise in many fields such as data analysis (classification,
text mining) \cite{Bon:11,ChaSin:08}, systems and control theory
(optimal control, pole assignment by static output feedback)
\cite{FaiMar:12,JudRay:08,NecCli:13,Par:97}, machine learning
\cite{ChaSin:08,NecPat:12,RicTac:13,RicTac:13_2,ShaZha:13,Vap:95}
and truss topology design \cite{KocOut:06,RicTac:12}. The goal of
this paper is to analyze several  new random coordinate gradient
descent methods suited for large-scale nonconvex problems with
composite objective
function. 
Recently, after our paper came under review, a variant of random
coordinate descent method for solving composite nonconvex problems
was also proposed in  \cite{LuXia:13}.  For our coordinate descent
algorithm, which is designed to minimize unconstrained composite
nonconvex objective functions, we prove asymptotic convergence of
the generated sequence to stationary points and sublinear rate of
convergence in expectation for some optimality measure.
We also provide convergence
analysis for a coordinate descent method designed for solving singly
linearly constrained nonconvex  problems and obtain similar results
as in the unconstrained case. Note that our analysis is very
different from the convex case
\cite{Nec:13,NecPat:12,NecNes:12,Nes:10,RicTac:11} and is based on
the notion of optimality measure and a supermartingale convergence
theorem.  Furthermore, unlike to other coordinate descent methods
for nonconvex problems, our algorithms offer some important
advantages, e.g. due to the randomization our algorithms are simpler
and are adequate for modern computational architectures. We also
present the results of preliminary computational experiments, which
confirm the superiority of our methods compared with other
algorithms for large-scale nonconvex optimization.

\noindent \textbf{Contribution}. The contribution of the paper can
be summarized as follows:\vspace*{-0.1cm}
\begin{enumerate}
\item[(a)] For unconstrained problems we propose a 1-random coordinate descent method (1-RCD),
that involves at each iteration the solution of an optimization
subproblem with respect to only one (block) variable while keeping
all others fixed. We show that this solution can be usually computed
in closed form (Section \ref{randcoordmet}).

\item[(b)] For the linearly constrained case we propose a
 2-random coordinate descent method (2-RCD),
that involves at each iteration the solution of
 a subproblem depending on two (block) variables while keeping all other
variables fixed. We show that in most cases this solution can be
found in linear time (Section \ref{2-randcoordmet}).

\item[(c)] For each of the algorithms we introduce some optimality measure and devise a
convergence analysis  using this framework. In particular, for both
algorithms, (1-RCD) and (2-RCD), we establish asymptotic convergence
of the generated sequences to stationary points (Theorems
\ref{thconv} and \ref{lemma3}) and sublinear rate of convergence for
the expected values of the corresponding optimality measures
(Theorems \ref{thcr} and \ref{constrainedrate}).
\end{enumerate}

\noindent \textbf{Content}. The structure of the paper is as
follows. In Section 2 we introduce a 1-random coordinate descent
algorithm for unconstrained minimization of nonconvex composite
functions. Further, we analyze the convergence properties of the
algorithm under standard assumptions. In Section 3 we derive a 2-random coordinate
descent method for solving singly linearly constrained nonconvex
problems and analyze  its convergence. In Section 4 we report
numerical results on large-scale eigenvalue complementarity
problems, which is an important application in  control theory.



\noindent \textbf{Notation}. We consider the space $\rset^n$
composed by column vectors. For $x,y \in \rset^n$ we denote the
scalar product by $\langle x,y \rangle = x^T y$ and
$\norm{x}=(x^Tx)^{1/2}$. We use the same notation $\langle
\cdot,\cdot \rangle$ and $\norm{\cdot}$ for scalar products and
norms in spaces of different dimensions. For some norm
$\norm{\cdot}_\alpha$ in $\rset^n$, its dual norm is defined by
$\norm{y}^*_\alpha= \max_{\norm{x}_\alpha=1} \langle y,x \rangle$.
We consider the following decomposition of the variable dimension:
$n = \sum_{i=1}^N n_i$. Also, we denote a block decomposition of $n
\times n$ identity matrix by $I_n= \left[ U_1 \dots U_N \right]$,
where $U_i \in \rset^{n \times n_i}$. For brevity we use the
following notation: for all $x \in \rset^n$ and $i, j = 1, \dots,
N$, we denote:
\begin{align*}
 x_i &= U_i^T x \in \rset^{n_i}, \quad \quad \quad \quad \ \
 \nabla_i f(x)= U_i^T \nabla f(x) \in \rset^{n_i} \\
 x_{ij}&=\left[x_i^T \ x_j^T \right]^T \in \rset^{n_i+n_j},
 \quad \nabla_{ij} f(x)= \left[ \nabla_i f(x)^T \ \nabla_j f(x)^T\right]^T \in \rset^{n_i+n_j}.
\end{align*}


\section{Unconstrained minimization of composite objective functions}
\label{sectionRCD}

In this section we analyze  a variant of random block coordinate
gradient descent method, which we call the \textit{1-random
coordinate descent method} (1-RCD), for solving large-scale
unconstrained nonconvex problems with composite objective function.
The method  involves at each iteration the solution of an
optimization subproblem only with respect to one (block) variable
while keeping all other variables fixed. After discussing several
necessary mathematical preliminaries, we introduce  an optimality
measure, which will be the basis for the construction and analysis
of Algorithm (1-RCD). We establish asymptotic convergence of the
sequence generated by Algorithm (1-RCD) to a stationary point  and
then we show sublinear rate of convergence in expectation for the
corresponding optimality measure. For some well-known particular
cases of nonconvex objective functions arising frequently in
applications, the complexity per iteration  of our  Algorithm
(1-RCD) is of order $\mathcal{O}(\max\limits_{i} n_i)$.

\subsection{Problem formulation}

The problem of interest in this section is the unconstrained nonconvex
minimization problem with composite objective function:
\begin{equation}\label{model}
 F^* = \min\limits_{x \in \rset^n} F(x) \ \left(:= f(x) + h(x)\right),
\end{equation}
where the function $f$ is smooth and $h$ is a convex, separable,
nonsmooth function. Since $h$ is nonsmooth, then for any $x \in
dom(h)$ we denote by $\partial h(x)$ the subdifferential (set of
subgradients) of $h$ at $x$. The smooth and nonsmooth components in
the objective function of \eqref{model} satisfy the following
assumptions:
\begin{assumption}
\label{mainassumption}
\begin{enumerate}
\item[(i)]
The function $f$ has block coordinate Lipschitz continuous gradient,
i.e. there are  constants $L_{i}>0$ such that:
\begin{equation*}
 \norm{\nabla_{i} f(x+ U_is_{i}) - \nabla_{i}f(x)}  \le L_{i}
 \norm{s_{i}} \qquad \forall s_{i} \in \rset^{n_i}, \; x \in
 \rset^n, \; i=1,\dots,N.
\end{equation*}
\item[(ii)]
The function $h$ is  proper, convex, continuous and block separable:
\begin{equation*}
h(x) = \sum\limits_{i=1}^N h_i(x_i) \quad  \forall x \in \rset^n,
\end{equation*}
where  the functions $h_i: \rset^{n_i} \to \rset^{}$ are convex for
all $i=1, \dots, N$.
\end{enumerate}
\end{assumption}
 These assumptions are typical for the coordinate descent framework
and the reader can find similar variants in
\cite{Nec:13,NecPat:12,Nes:10,LuXia:13,TseYun:09}. An immediate
consequence of Assumption \ref{mainassumption} (i) is the following
well-known inequality \cite{Nes:04}:
\begin{equation}\label{coupledlemma}
 \abs{f(x + U_is_{i})  - f(x) - \langle \nabla_{i}f(x), s_{i} \rangle}
  \le \frac{L_{i}}{2}\norm{s_{i}}^2
 \quad \forall s_{i} \in \rset^{n_i},\;  x \in \rset^n.
\end{equation}
Based on this quadratic approximation of function $f$ we get the
inequality:
\begin{equation}\label{objquad}
 F(x+U_is_i)\le f(x)+\langle \nabla_{i} f(x),s_{i} \rangle+\frac{L_{i}}{2} \norm{s_{i}}^2
 +h(x+U_is_i) \quad \forall s_{i} \in \rset^{n_i}, \; x \in \rset^n.
\end{equation}

\noindent Given  local Lipschitz constants $L_i>0$ for $i=1, \dots,
N$, we define the vector $L=[L_1 \dots L_N]^T \in \rset^N$, the
diagonal matrix $D_L = \text{diag}(L_1 I_{n_1}, \dots, L_N I_{n_N})
\in \rset^{n \times n}$ and the following pair of dual norms:
\begin{equation*}
\norm{x}_L = \left( \sum_{i=1}^N L_i \norm{x_i}^2 \right)^{1/2} \;\;
\forall x \in \rset^n,  \quad \norm{y}_L^* = \left( \sum_{i=1}^N
L_i^{-1} \norm{y_i}^2 \right)^{1/2}  \;\; \forall y \in \rset^n.
\end{equation*}

Under Assumption \ref{mainassumption}, we can state  the first order
necessary optimality conditions for the nonconvex optimization
problem \eqref{model}: if $x^* \in \rset^n$ is a local minimum for
\eqref{model}, then the following  relation holds
\begin{equation*}
0 \in \nabla f(x^*) + \partial h(x^*).
\end{equation*}
Any vector $x^*$ satisfying this relation is called a
\textit{stationary point} for nonconvex problem \eqref{model}.


\subsection{A 1-random coordinate descent algorithm}
\label{randcoordmet}

We analyze a variant of random coordinate descent method suitable
for solving large-scale nonconvex problems of the form
\eqref{model}. Let $i \in \{1, \dots, N\}$ be a random variable and
$p_{i_k} = \text{Pr}(i=i_k)$ be its probability distribution. Given
a point $x$, one block is chosen randomly with respect to the
probability distribution $p_i$ and the quadratic model
\eqref{objquad} derived from the composite objective function is
minimized with respect to this block of coordinates (see also
\cite{Nes:10,RicTac:11}). Our method has the following iteration:
given an initial point $x_0$, then for all $k \geq 0$
\begin{equation*}\label{algorithm}
\boxed{
\begin{split}
& \textbf{Algorithm (1-RCD)}\\
&1.\ \text{Choose randomly a block of coordinates}\  i_k  \
\text{with
probability}\  p_{i_k}\\
 &2.\ \text{Set} \ x^{k+1} = x^k + U_{i_k} d_{i_k},
 \end{split}}
 \end{equation*}
where the direction $d_{i_k}$ is computed as follows:
\begin{equation}\label{extdir}
d_{i_k} = \arg \min_{s_{i_k}\in \rset^{n_{i_k}}} f(x^k) +
\langle\nabla_{i_k} f (x^k), s_{i_k} \rangle +
\frac{L_{i_k}}{2}\norm{s_{i_k}}^2 + h(x^k + U_{i_k} s_{i_k}).
\end{equation}
Note that the  direction $d_{i_k}$ is a minimizer of the quadratic
approximation model  given in \eqref{objquad}. Further, from
Assumption \ref{mainassumption} (ii) we see that $ h(x^k + U_{i_k}
s_{i_k}) = h_{i_k} (x_{i_k}^k + s_{i_k}) + \sum_{i \not = i_k}
h_i(x_i^k)$ and thus for computing $d_{i_k}$ we only need to know
the function $h_{i_k}(\cdot)$. An important property of our
algorithm is that for certain particular cases of function $h$, the
complexity per iteration of Algorithm (1-RCD) is very low. In
particular, for certain simple functions $h$, very often met in many
applications from signal processing, machine learning and optimal
control, the direction $d_{i_k}$ can be computed in closed form,
e.g.:
\begin{enumerate}
 \item[(I)]
For some $l, u \in \rset^n$, with $l \le u$, we consider the box
indicator function
\begin{align}
\label{boxfct} & h(x) =\begin{cases}
          0  &  \text{if} \;\;  l \le x \le u\\
       \infty           & \text{otherwise}.
       \end{cases}
\end{align}
In this case the direction $d_{i_k}$ has the explicit  expression:
\[ d_{i_k} =\left[x_{i_k}^k - \frac{1}{L_{i_k}}\nabla_{i_k}
f(x^k) \right]_{[l_{i_k},\ u_{i_k}]}  \quad \forall i_k=1,\dots,N,\]
where $[x]_{[l, \ u]}$ is the orthogonal projection of vector $x$ on
box set $[l, \ u]$.

\item [(II)]
 Given a nonnegative scalar $\beta \in \rset^{}_{+}$, we consider
 the $\ell_1$-regularization function defined by the 1-norm
\begin{equation} \label{regfct}
h(x) = \beta \norm{x}_1.
\end{equation}
In this case, considering  $n=N$,  the direction $d_{i_k}$ has the
explicit expression:
\[ d_{i_k} = \text{sgn}(t_{i_k}) \cdot \max \left\{\abs{t_{i_k}} -
\frac{\beta}{L_{i_k}}, \; 0 \right\} - x_{i_k}  \quad \forall
i_k=1,\dots,n,
 \]
where $t_{i_k}= x_{i_k} - \frac{1}{L_{i_k}}\nabla_{i_k}f(x^k)$.
\end{enumerate}
In  these examples the arithmetic complexity of computing the next
iterate $x^{k+1}$, once $\nabla_{i_k}f(x^k)$ is known, is of order
$\mathcal{O}(n_{i_k})$.  The reader can find other favorable
examples of nonsmooth functions $h$ which preserve the low iteration
complexity of Algorithm (1-RCD) (see also \cite{NecPat:12,TseYun:09}
for other examples). Note that most of the  (coordinate descent)
methods designed for solving nonconvex problems usually have
complexity per iteration at least of order $\mathcal{O}(n)$ (see
e.g. \cite{TseYun:09},  where the authors analyze a greedy
coordinate descent method). Coordinate descent methods that have
similar complexity per iteration as our random method can be found
e.g. in \cite{TseYun:09_2}, where the index selection is made
cyclically (Gauss-Seidel rule). But Algorithm (1-RCD) also offers
other important advantages, e.g. due to the randomization the
algorithm is adequate for modern computational architectures (e.g
distributed and parallel architectures) \cite{NecCli:13,RicTac:13}.

We assume that the sequence of random variables $i_0, \dots, i_k$
are i.i.d. In the sequel, we use the notation $\xi^k$ for the entire
history of random index selection
\begin{equation*}
    \xi^k = \left\lbrace i_0, \dots, i_{k} \right\rbrace
\end{equation*}
and notation  \[ \phi^k = E \left[ F(x^k) \right], \] for the
expectation taken w.r.t. $\xi^{k-1}$. Given $s, x \in \rset^n$, we
introduce the following function and the associated map (operator):
\begin{align}
&\psi_L(s;x) = f(x) + \langle \nabla f(x), s\rangle + \frac{1}{2}\norm{s}_L^2 + h(x+s),\nonumber\\
& d_L(x) = \arg\min\limits_{s \in \rset^n} f(x) + \langle \nabla f(x),
s\rangle + \frac{1}{2}\norm{s}_L^2 + h(x+s) \label{dirtotal}.
\end{align}
Based on this map, we now introduce an \textit{optimality measure}
which will be the basis for the analysis of Algorithm (1-RCD):
\begin{equation*}
 M_1(x,L) = \norm{D_L \cdot d_L(x)}_L^*.
\end{equation*}

\noindent The map $M_1(x,L)$ is an optimality measure for
optimization problem \eqref{model} in the sense that it is positive
for all nonstationary points and zero for stationary points (see
Lemma \ref{lemmaconv} below):
\begin{lemma}
\label{lemmaconv} For any given vector $\tilde{L} \in \rset^N$ with
positive entries, a vector $x^* \in \rset^n$ is a stationary point
for problem \eqref{model} if and only if the value
$M_1(x^*,\tilde{L})=0$.
\end{lemma}
\begin{proof}: Based on  the optimality conditions of subproblem \eqref{dirtotal},
it can be easily shown that if $M_1(x^*,\tilde{L})=0$, then  $x^*$ is a
stationary point for the original problem \eqref{model}. We prove
the converse implication  by contradiction. Assume that $x^*$ is a
stationary point for \eqref{model} and $M_1(x^*,\tilde{L})>0$. It follows
that $d_{\tilde{L}}(x^*)$ is a nonzero solution of subproblem
\eqref{dirtotal}. Then, there exist the subgradients $g(x^*) \in
\partial h(x^*)$ and  $g(x^* + d_{\tilde{L}}(x^*)) \in
\partial h(x^* + d_{\tilde{L}}(x^*))$ such that the optimality conditions for optimization
problems \eqref{model} and \eqref{dirtotal} can be written as:
\begin{equation*}
\begin{cases}
 \nabla f(x^*) + g(x^*)= 0 \\
 \nabla f(x^*) + D_{\tilde{L}} d_{\tilde{L}}(x^*) + g(x^* + d_{\tilde{L}}(x^*))=0.
\end{cases}
\end{equation*}
Taking the difference of the two relations above and considering the
inner product with $d_{\tilde{L}}(x^*) \not=0$ on both sides of the equation,
we get:
\begin{equation*}
\norm{d_{\tilde{L}}(x^*)}_{\tilde{L}}^2 + \langle g(x^* + d_{\tilde{L}}(x^*)) -
g(x^*), d_{\tilde{L}}(x^*) \rangle =0.
\end{equation*}
From convexity of the function $h$ we see that both terms in the
above sum are nonnegative and thus  $d_{\tilde{L}}(x^*) =0$, which contradicts
our hypothesis. In conclusion $M_1(x^*,\tilde{L})=0$. \qed
\end{proof}

Note that $\psi_L(s;x)$ is an  $1$-strongly convex function in the
variable $s$ w.r.t. norm $\norm{\cdot}_L$ and thus $d_L(x)$ is
unique and the following inequality holds:
\begin{equation}\label{strongconvex}
\psi_L(s;x) \ge \psi_L(d_L(x);x) + \frac{1}{2} \norm{d_L(x)-s}_L^2
\quad  \forall x, s \in \rset^n.
\end{equation}


\subsection{Convergence  of Algorithm (1-RCD)}

In this section, we analyze the  convergence properties of Algorithm
(1-RCD). Firstly, we prove the asymptotic convergence of the
sequence generated by  Algorithm (1-RCD) to stationary points. For
proving the asymptotic convergence we use the following
supermartingale convergence result of Robbins and Siegmund (see
\cite[Lemma 11 on page 50]{Pol:87}):
\begin{lemma}
\label{mart} Let $v_k, u_k$ and $\alpha_k$ be three sequences of
nonnegative random variables such that
\[ E[v_{k+1} | {\cal F}_k] \leq (1+\alpha_k) v_k - u_k
\;\; \forall k \geq 0 \; \text{a.s.} \;\;\;  \text{and} \;\;\;
\sum_{k=0}^\infty \alpha_k < \infty \;  \text{a.s.}, \]  where
${\cal F}_k$ denotes the collections $v_0, \dots, v_k, u_0, \dots,
u_k$, $\alpha_0, \dots, \alpha_k$. Then, we have $\lim_{k \to
\infty} v_k = v$ for a random variable $v \geq 0$ a.s. and \;
$\sum_{k=0}^\infty u_k < \infty$ a.s.
\end{lemma}
In the next lemma we prove that Algorithm (1-RCD) is a descent
method, i.e. the objective function is nonincreasing along its
iterations:
\begin{lemma}
\label{lemadec} Let  $x^k$  be the sequence generated by Algorithm
(1-RCD)  under Assumption \ref{mainassumption}. Then, the following
relation holds:
\begin{align}
\label{decr} F(x^{k+1}) \le F(x^k) -\frac{L_{i_k}}{2}
\norm{d_{i_k}}^2 \quad \forall k \geq 0.
\end{align}
\end{lemma}
\begin{proof}:
From the optimality conditions of subproblem \eqref{extdir} we have
that there exists a subgradient $g(x^k_{i_k} + d_{i_k}) \in
\partial h_{i_k}(x^k_{i_k} + d_{i_k})$ such that:
\[  \nabla_{i_k} f (x^k) + L_{i_k}d_{i_k} + g(x^k_{i_k} + d_{i_k}) =0.  \]
On the other hand, since  the function $h_{i_k}$ is convex,
according to Assumption \ref{mainassumption} (ii), the following
inequality holds:
\[ h_{i_k}(x^k_{i_k} + d_{i_k}) - h_{i_k}(x^k_{i_k}) \leq
 \langle g(x^k_{i_k} + d_{i_k}), d_{i_k} \rangle \]

Applying the previous two relations in  \eqref{objquad} and using
the separability of the function $h$, then under Assumption
\ref{mainassumption} (ii) we have that
\begin{align*}
 F(x^{k+1}) & \le  F(x^k) + \langle \nabla_{i_k} f(x^k),d_{i_k} \rangle +
 \frac{L_{i_k}}{2} \norm{d_{i_k}}^2
 +h_{i_k}(x^k_{i_k} + d_{i_k}) - h_{i_k}(x^k_{i_k}) \\
 &\le F(x^k) +  \langle \nabla_{i_k} f(x^k),d_{i_k} \rangle +
 \frac{L_{i_k}}{2} \norm{d_{i_k}}^2 + \langle g(x^k_{i_k} + d_{i_k}), d_{i_k} \rangle
 \\
 &\le F(x^k)- \frac{L_{i_k}}{2}
\norm{d_{i_k}}^2.
\end{align*}
\qed
\end{proof}

Using Lemma \ref{lemadec}, we state the following result regarding
the asymptotic convergence  of Algorithm (1-RCD).

\begin{theorem}
\label{thconv}
 If Assumption \ref{mainassumption} holds for the composite objective function $F$
 of problem \eqref{model} and the sequence $x^k$ is generated by Algorithm (1-RCD)
 using the uniform distribution, then the following statements are valid:
\begin{enumerate}
\item[(i)] The sequence of random variables $M_1(x^k,L)$ converges  to 0 a.s.
and the sequence $F(x^k)$ converges to a random variable $\bar F$ a.s.
\item[(ii)] Any accumulation point of the sequence $x^k$ is a stationary
point for optimization problem~\eqref{model}.
\end{enumerate}
\end{theorem}

\begin{proof}
(i) From Lemma \ref{lemadec} we get:
\begin{equation*}
F(x^{k+1}) - F^* \le F(x^k) - F^*  -
\frac{L_{i_k}}{2}\norm{d_{i_k}}^2 \quad \forall k \geq 0.
\end{equation*}
We now take the expectation  conditioned on $\xi^{k-1}$ and note
that $i_k$ is independent on the past $\xi^{k-1}$, while $x^k$ is
fully determined by $\xi^{k-1}$. We thus obtain:
\begin{align*}
E \left[ F(x^{k+1}) - F^*| \; \xi^{k-1} \right]  & \le F(x^ k) - F^*
- \frac{1}{2} E\left[L_{i_k} \cdot \norm{d_{i_k}}^2 | \;  \xi^{k-1} \right] \\
&\le  F(x^k) - F^*  - \frac{1}{2N}\norm{d_L(x^k)}_L^2.
\end{align*}
Using the supermartingale convergence theorem given in Lemma
\ref{mart} in the previous inequality, we can ensure that \[ \lim_{k
\to \infty} F(x^k) - F^* = \theta \quad \text{a.s.}\] for a random
variable $\theta \geq 0$ and thus $\bar F=\theta + F^*$. Further,
due to almost sure convergence of sequence $F(x^k)$, it can be
easily seen that $\lim\limits_{k \to \infty} F(x^k) - F(x^{k+1})=0$
a.s. From $x^{k+1} - x^k = U_{i_k}d_{i_k}$ and  Lemma \ref{lemadec}
we have:
\begin{equation*}
\frac{L_{i_k}}{2}\norm{ d_{i_k} }^2=
\frac{L_{i_k}}{2}\norm{x^{k+1}-x^k}^2 \le F(x^k) - F(x^{k+1})  \quad
\forall k \geq 0,
\end{equation*}
which implies  that \[ \lim\limits_{k \to \infty}
\norm{x^{k+1}-x^k}=0 \quad \text{and} \quad  \lim\limits_{k \to
\infty} \norm{d_{i_k}} = 0 \quad \text{a.s.}
\]  As $\norm{d_{i_k}} \to 0$ a.s., we can conclude that the random
variable $E[\norm{d_{i_k}}| \xi^{k-1}] \to 0$ a.s. or equivalently
$M_1(x^k,L) \to 0$ a.s.

(ii) For brevity we assume that the entire sequence $x^k$ generated
by Algorithm (1-RCD) is convergent. Let $\bar{x}$  be the limit
point of the sequence $x^k$. In the first part of the theorem we
proved that the sequence of random variables $d_L(x^k)$ converges to
$0$ a.s. Using the definition of $d_L(x^k)$ we have:
\begin{align*}
f(x^k)& +\langle \nabla f(x^k), d_L(x^k)\rangle +
\frac{1}{2}\norm{d_L(x^k)}_L^2 + h(x^k +d_L(x^k))\\
& \le  f(x^k) +\langle \nabla f(x^k), s\rangle +
\frac{1}{2}\norm{s}_L^2 + h(x^k+s) \quad \forall s \in \rset^n,
\end{align*}
and  taking the limit $k \to \infty$ and using Assumption
\ref{mainassumption} (ii) we get:
\begin{equation*}
F(\bar{x})\le f(\bar{x}) +\langle \nabla f(\bar{x}), s\rangle +
\frac{1}{2}\norm{s}_L^2 + h(\bar{x}+s) \quad \forall s \in \rset^n.
\end{equation*}
This shows that $d_L(\bar{x}) = 0$ is the minimum in subproblem
\eqref{dirtotal} for $x = \bar{x}$ and  thus $M_1(\bar{x},L) = 0$.
From  Lemma \ref{lemmaconv} we conclude that  $\bar{x}$ is a
stationary point for optimization problem \eqref{model}.  \qed
\end{proof}

\noindent The next theorem proves the  convergence rate of the
optimality measure $M_1(x^k,L)$ towards $0$ in expectation.
\begin{theorem}
\label{thcr} Let $F$ satisfy Assumption \ref{mainassumption}. Then,
the Algorithm (1-RCD) based on the uniform distribution generates a
sequence $x^k$ satisfying the following convergence rate for the
expected values of the optimality measure:
\begin{equation*}
\min\limits_{0 \le l \le k} E \left[ \left( M_1(x^l,L) \right)^2
\right] \le \frac{2N\left(F(x^0) - F^*\right)}{k+1} \qquad \forall k
\geq 0.
\end{equation*}
\end{theorem}

\begin{proof}: For simplicity of the exposition we use the following notation:
given the current iterate $x$, denote the next iterate $x^+ = x + U_i d_i$,
where direction $d_i$ is given by \eqref{extdir} for some
random chosen index $i$ w.r.t. uniform distribution. For brevity, we
 also adapt  the notation of expectation upon the entire history, i.e.
$(\phi, \phi^+, \xi)$ instead of $(\phi^k, \phi^{k+1}, \xi^{k-1})$.
From  Assumption \ref{mainassumption} and inequality \eqref{objquad}
we have:
\begin{align*}
F(x^+) & \le  f(x) + \langle \nabla_i f(x), d_i \rangle +
\frac{L_i}{2}\norm{d_i}^2 + h_i(x_i + d_i) + \sum\limits_{j\neq i} h_j(x_j).
\end{align*}
We now take the expectation  conditioned on $\xi$:
\begin{align*}
 &E [F(x^+)| \ \xi ]  \le\!  E \Big[ f(x) \!+\! \langle \nabla_i f(x), d_i
\rangle \!+\! \frac{L_i}{2} \norm{d_i}^2 + h_i(x_i + d_i) \!+\!
\sum\limits_{j \neq i} h_j(x_j) | \; \xi
 \Big]\\
& \le f(x)  + \frac{1}{N} \Big[\langle \nabla f(x), d_{L}(x)\rangle
 + \frac{1}{2} \norm{d_{L}(x)}_L^2  + h(x+d_{L}(x)) + (N-1)h(x) \Big].
\end{align*}
After rearranging the above expression we get:
\begin{equation}\label{compact}
 E[ F(x^+)| \; \xi ] \le \left(1-\frac{1}{N}\right)F(x)+\frac{1}{N}\psi_L(d_L(x);x).
\end{equation}
Now, by taking the expectation in \eqref{compact} w.r.t. $\xi$ we
obtain:
\begin{equation}\label{explips}
 \phi^+ \le \left(1-\frac{1}{N}\right)\phi + E \left[ \frac{1}{N} \psi_{L}(d_L(x);x) \right],
\end{equation}
and then using the $1-$strong convexity property  of $\psi_{L}$ we
get:
\begin{align}
\phi - \phi^+ & \ge \phi - \left(1-\frac{1}{N}\right)\phi - \frac{1}{N}
E\left[\psi_L(d_L(x);x) \right]\nonumber \\
& = \frac{1}{N} \left( E \left[ \psi_L(0;x)]  - E [\psi_L(d_L(x);x) \right] \right)\nonumber \\
& \ge \frac{1}{2N} E \left[ \norm{d_L(x)}_L^2 \right]= \frac{1}{2N}
E \left[ (M_1(x,L))^2 \right]. \label{descent}
\end{align}

\noindent Now coming back to the notation  dependent on  $k$ and
summing w.r.t. the entire history we have:
\begin{equation*}
  \frac{1}{2N} \sum\limits_{l=0}^k E \left[ (M_1(x^l,L))^2 \right] \le \phi^0 -
  F^*,
\end{equation*}
which leads to the statement of the theorem. \qed
\end{proof}

It is important to note that the convergence rate for the Algorithm
(1-RCD) given in Theorem \ref{thcr} is typical for the class of
first order methods designed for solving nonconvex and nonsmooth
optimization problems (see e.g. \cite{Nes:07} for more details).
Recently, after our paper came under review, a variant of 1-random
coordinate descent method for solving composite nonconvex problems
was also proposed in  \cite{LuXia:13}. However, the authors in
\cite{LuXia:13} do not provide complexity results for their
algorithm, but only asymptotic convergence in expectation. Note also
that our convergence results are different from the convex case
\cite{Nes:10,RicTac:11}, since here we introduce another optimality
measure and we use the supermartingale convergence theorem in the
analysis.

Furthermore, when the objective function $F$ is smooth and
nonconvex, i.e. $h=0$, the first order necessary conditions of
optimality become $\nabla f(x^*) =0$. Also, note that in this case,
the optimality measure $M_1(x,L)$ is given by: $M_1(x,L) =
\norm{\nabla f(x)}^*_L$. An immediate consequence of Theorem
\ref{thcr} in this case is the following result:
\begin{corrolary}
Let $h = 0$ and $f$ satisfy Assumption \ref{mainassumption} (i).
Then, in this case, the Algorithm (1-RCD) based on the uniform
distribution generates a sequence $x^k$ satisfying the following
convergence rate for the expected values of the norm of the
gradients:
\begin{equation*}
\min\limits_{0 \le l \le k} E \left[ \left(\norm{\nabla f(x^l)}_L^*
\right)^2 \right] \le  \frac{2N \left( F(x^0) - F^* \right)}{k+1}
\quad \forall k \geq 0.
\end{equation*}
\end{corrolary}


\section{Constrained minimization of composite objective functions}

In this section we present a variant of random block coordinate
gradient descent method for solving large-scale nonconvex
optimization problems with composite objective function and a single
linear equality constraint:
\begin{align}\label{modelcon}
 F^* = \min\limits_{x \in \rset^n} & \ F(x) \ \ \left(:=f(x)+h(x)\right)  \\
 \text{s.t.:} \;\;  & a^T x = b,    \nonumber
\end{align}
where $a \in \rset^n$ is a nonzero vector and functions $f$ and $h$
satisfy similar conditions as in Assumption \ref{mainassumption}. In
particular, the smooth and nonsmooth part of the objective function
in \eqref{modelcon} satisfy:
\begin{assumption}
\label{mainassumption2}
\begin{enumerate}
\item[(i)]
The function $f$ has 2-block coordinate Lipschitz continuous
gradient, i.e. there are constants $L_{ij}>0$ such that:
\begin{equation*}
 \norm{\nabla_{ij} f(x+U_i s_{i} + U_j s_j) - \nabla_{ij}f(x)}  \le L_{ij}
 \norm{s_{ij}}
\end{equation*}
for all $s_{ij} =[s_i^T \; s_j^T]^T \in \rset^{n_i+n_j}$,
 $ x \in \rset^n$ and $i, j =1,\dots,N$.
\item[(ii)]
The function $h$ is  proper, convex, continuous and coordinatewise
separable:
\begin{equation*}
h(x) = \sum\limits_{i=1}^n h_i(x_i)  \quad  \forall x \in \rset^n,
\end{equation*}
where the functions $h_i:\rset^{} \to \rset^{}$ are convex for all
$i=1,\dots,n$.
\end{enumerate}
\end{assumption}
Note that these assumptions are frequently used in the area of
coordinate descent methods for convex minimization, e.g.
\cite{Bec:12,Nec:13,NecPat:12,NecNes:12,TseYun:09}. Based on this
assumption the first order necessary optimality conditions become:
if $x^*$ is a local minimum of \eqref{modelcon}, then  there exists
a scalar $\lambda^*$ such that:
\[
0 \in \nabla f(x^*) + \partial h(x^*)+ \lambda^* a \quad \text{and}
\quad a^T x^*=b. \] Any vector $x^*$ satisfying this relation is
called a \textit{stationary point} for nonconvex problem
\eqref{modelcon}. For a simpler exposition in the following sections
we use a context-dependent notation as follows: let  $x =
\sum_{i=1}^N U_i x_i \in \rset^n$ and  $x_{ij} =[x_i^T \; x_j^T]^T
\in \rset^{n_i + n_j}$, then  by addition with a vector in the
extended space $y\in \rset^n$, i.e.,  $y+x_{ij}$, we understand $y+
U_ix_i+U_jx_j$. Also, by the inner product $\langle y,
x_{ij}\rangle$ we understand $\langle y, x_{ij}\rangle = \langle
y_i, x_i\rangle  + \langle y_j, x_j\rangle$. Based on Assumption
\ref{mainassumption2} (i)  the following inequality holds
\cite{NecPat:12}:
\begin{equation}\label{lipsij}
 |f(x+s_{ij}) -  f(x) + \langle \nabla_{ij} f(x), s_{ij} \rangle | \leq  \frac{L_{ij}}{2}
 \norm{s_{ij}}^2 \quad \forall x \in \rset^n, \; s_{ij} \in \rset^{n_i + n_j}
\end{equation}
and then we can bound the function $F$ with the following quadratic
expression:
\begin{equation}
\label{objquadij}
 F(x+s_{ij}) \le f(x) + \langle \nabla_{ij} f(x),s_{ij} \rangle +
 \frac{L_{ij}}{2} \norm{s_{ij}}^2
 +h(x+s_{ij}) \;\; \forall s_{ij} \in \rset^{n_i+n_j}, x \in \rset^n.
\end{equation}

\noindent Given  local Lipschitz constants $L_{ij}>0$ for $i\neq j
\in \{1, \dots, N\}$, we define the vector $\Gamma \in \rset^N$ with the
components   $\Gamma_i = \frac{1}{N}\sum\limits_{j=1}^N L_{ij}$, the
diagonal matrix $D_{\Gamma} = \text{diag}(\Gamma_1 I_{n_1}, \dots, \Gamma_N
I_{n_N}) \in \rset^{n \times n}$ and the following pair of dual
norms:
\begin{equation*}
 \norm{x}_\Gamma = \left( \sum_{i=1}^N \Gamma_i \norm{x_i}^2 \right)^{1/2}  \; \forall x \in \rset^n,
  \quad \norm{y}_\Gamma^* = \left( \sum_{i=1}^N \Gamma_i^{-1} \norm{y_i}^2 \right)^{1/2}
  \;
\forall y \in \rset^n.
\end{equation*}


\subsection{A 2-random coordinate descent algorithm}
\label{2-randcoordmet}

Let $(i,j)$ be a two dimensional random variable, where $i, j \in
\{1, \dots, N\}$ with $ i \neq j$ and
$p_{i_kj_k}=\text{Pr}((i,j)=(i_k,j_k))$ be its probability
distribution. Given a feasible $x$, two blocks are chosen randomly
with respect to a given  probability distribution $p_{ij}$ and the
quadratic model \eqref{objquadij}  is minimized with respect to
these coordinates. Our method has the following iteration: given a
feasible initial point $x^0$, that is $a^Tx^0=b$, then for all $k
\geq 0$
\begin{equation*}\label{algorithm}
\boxed{
\begin{split}
& \textbf{Algorithm (2-RCD)}\\
&1. \ \text{Choose randomly 2 block coordinates} \;
(i_k,j_k) \;   \text{with probability} \;  p_{i_kj_k} \\
 &2.\ \text{Set} \ x^{k+1} = x^k + U_{i_k} d_{i_k} + U_{j_k} d_{j_k},
 \end{split}
 }
 \end{equation*}
where directions $d_{i_kj_k} =[d_{i_k}^T \; d_{j_k}^T]^T$
minimize the quadratic model \eqref{objquadij}:
\begin{align}
d_{i_kj_k} = \arg &\min_{s_{i_kj_k}} f(x^k) + \langle
\nabla_{i_kj_k} f (x^k), s_{i_kj_k} \rangle + \frac{L_{i_kj_k}}{2}
\norm{s_{i_kj_k}}^2 +  h(x^k +s_{i_k j_k}) \nonumber \\
\text{s.t.:} & \quad a_{i_k}^T s_{i_k} + a_{j_k}^T s_{j_k}=0.
\label{dir2}
\end{align}
The reader should note  that for problems with simple separable
functions $h$ (e.g. box indicator function \eqref{boxfct},
$\ell_1$-regularization function \eqref{regfct}) the arithmetic
complexity of computing  the direction $d_{ij}$ is
$\mathcal{O}(n_i+n_j)$ (see \cite{NecPat:12,TseYun:09} for a
detailed discussion). Moreover, in the scalar case, i.e. when $N=n$,
the search direction $d_{ij}$ can be computed in closed form,
provided that $h$ is simple (e.g. box indicator function or
$\ell_1$-regularization function) \cite{NecPat:12}.  Note that other
(coordinate descent) methods designed for solving nonconvex problems
subject to a single linear equality constraint have complexity per
iteration at least of order $\mathcal{O}(n)$
\cite{Bec:12,LinLuc:09,ThiMoe:10,TseYun:09}. We can consider more
than one equality constraint in the optimization model
\eqref{modelcon}. However, in this case  the analysis of  Algorithm
(2-RCD) is involved and the complexity per iteration is much higher
(see \cite{NecPat:12,TseYun:09} for a detailed discussion).

We assume that  for every pair $(i,j)$ we have $p_{ij}=p_{ji}$ and
$p_{ii}=0$, resulting in $\frac{N(N-1)}{2}$ different pairs $(i,j)$.
We define the subspace $S= \{s \in \rset^n: \; a^T s = 0\}$  and the
local subspace w.r.t. the pair $(i,j)$ as $S_{ij} = \{x \in S: \; \;
x_l = 0 \;\; \forall l \neq i,j \}$.  Also, we denote $ \xi^k = \{
(i_0,j_0), \dots, (i_{k},j_{k})\} $ and  $ \phi^k = E \left[ F(x^k)
\right] $ for the expectation taken w.r.t. $\xi^{k-1}$. Given a
constant $\alpha>0$ and a vector with positive entries $L \in
\rset^N$, the following property is valid for $\psi_L$:
\begin{equation}\label{psialpha}
\psi_{\alpha L}(s;x) = f(x) + \langle \nabla f(x), s\rangle +
\frac{\alpha}{2}\norm{s}_L^2 + h(x+s).
\end{equation}
Since in this section we deal with linearly constrained problems, we
need to  adapt the definition for the map $d_L(x)$ introduced in
Section 2. Thus, for any vector with positive entries $L \in
\rset^N$ and $x \in \rset^n$, we define the following map:
\begin{equation}\label{dirtotal2}
d_L(x) = \arg\min\limits_{s \in S} f(x) + \langle \nabla f(x),
s\rangle + \frac{1}{2}\norm{s}_L^2 + h(x+s).
\end{equation}
In order to analyze the  convergence of Algorithm (2-RCD), we
introduce an \emph{optimality measure}:
\begin{equation*}
     M_2(x,\Gamma) = \norm{D_\Gamma \cdot d_{N\Gamma}(x)}^*_\Gamma.
\end{equation*}
\begin{lemma}

\label{lemmaconv2} For any given vector $\tilde \Gamma$ with positive
entries, a vector $x^* \in \rset^n$ is a stationary point for
problem \eqref{modelcon} if and only if the quantity $M_2(x^*,\tilde
\Gamma)=0$.
\end{lemma}
\begin{proof}: Based on  the optimality conditions of subproblem \eqref{dirtotal2},
it can be easily shown that if $M_2(x^*,\tilde \Gamma)=0$, then  $x^*$ is
a stationary point for the original problem \eqref{modelcon}. We
prove the converse implication  by contradiction. Assume that $x^*$
is a stationary point for \eqref{modelcon} and $M_2(x^*,\tilde \Gamma)>0$.
It follows that $d_{N\tilde{\Gamma}}(x^*)$ is a nonzero solution
of subproblem \eqref{dirtotal2} for $x = x^*$. Then, there exist the
subgradients $g(x^*) \in \partial h(x^*)$ and  $g(x^* +
d_{N\tilde{\Gamma}}(x^*)) \in \partial h(x^* + d_{N\tilde{\Gamma}}(x^*))$ and
two scalars $\gamma, \lambda \in \rset^{}$ such that the optimality
conditions for optimization problems \eqref{modelcon} and
\eqref{dirtotal2} can be written as:
\begin{equation*}
\begin{cases}
 \nabla f(x^*) + g(x^*)+ \lambda a= 0 \\
 \nabla f(x^*) + D_{N\tilde{\Gamma}} d_{N\tilde{\Gamma}}(x^*) + g(x^* + d_{N\tilde{\Gamma}}(x^*)) + \gamma a=0.
\end{cases}
\end{equation*}
Taking the difference of the two relations above and considering the
inner product with $d_{N\tilde{\Gamma}}(x^*) \not=0$ on both sides of the
equation, we get:
\begin{equation*}
\norm{d_{N\tilde{\Gamma}}(x^*)}_{\tilde \Gamma}^2 + \frac{1}{N}\langle g(x^* +
d_{N\tilde{\Gamma}}(x^*)) - g(x^*), d_{N\tilde{\Gamma}}(x^*) \rangle =0,
\end{equation*}
where we used that $a^T d_{N\tilde{\Gamma}}(x^*) =0$. From convexity
of the function $h$ we see that both terms in the above sum are
nonnegative and thus  $d_{N\tilde{\Gamma}}(x^*) =0$, which
contradicts our hypothesis. In conclusion, we get  $M_2(x^*,\tilde
\Gamma)=0$.  \qed
\end{proof}


\subsection{Convergence of Algorithm (2-RCD)}

In order to provide the convergence results of Algorithm (2-RCD), we
have to introduce some definitions and auxiliary results. We denote
by $\text{supp}(x)$ the set of indexes  corresponding to  the
nonzero coordinates in the vector $x \in \rset^n$.
\begin{definition}
Let $d, d' \in \rset^n$, then the vector $d'$ is {\it conformal} to
$d$ if: $ \text{supp}(d') \subseteq \text{supp}(d)$ and  $d'_j d_j
\ge 0$ for all $j=1, \dots, n$.
\end{definition}
We introduce the notion of elementary vectors for the linear
subspace $S = Null(a^T)$.
\begin{definition}
An elementary vector $d$ of $S$ is a vector $d \in S$ for which
there is no nonzero $d' \in S$ conformal to $d$ and
$\text{supp}(d')\neq \text{supp}(d)$.
\end{definition}

We now present some results for  elementary vectors and conformal
realization, whose proofs can be found  in
\cite{Roc:69,Roc:84,TseYun:09}. A particular case of Exercise 10.6
in \cite{Roc:84} and an interesting result in \cite{Roc:69} provide
us the following lemma:
\begin{lemma}\cite{Roc:69,Roc:84} \label{lemma1}
Given $d \in S$, if $d$ is an elementary vector, then
$\abs{\text{supp}(d)} \le 2$. Otherwise, $d$ has a conformal
realization  $d = d^1  + \dots + d^s$, where $s \ge 2$ and $d^t \in
S$ are elementary vectors conformal to $d$ for all $t=1, \dots, s$.
\end{lemma}
An important property of convex and separable functions is given by
the following lemma:
\begin{lemma} \label{lemma2} \cite{TseYun:09}
Let $h$ be componentwise separable and convex. For any  $x, x+d \in
\text{dom} h$, let $d$ be expressed as  $d = d^1+ \dots + d^s$ for
some $s \ge 2$ and some nonzero $d^t \in \rset^n$ conformal to $d$
for all $t=1, \dots, s$. Then,
\begin{equation*}
h(x+d) - h(x) \ge \sum\limits_{t=1}^s \left(h(x+d^t) - h(x)\right),
\end{equation*}
where  $d^t \in S$ are elementary vectors conformal to $d$ for all
$t=1, \dots, s$.
\end{lemma}

\begin{lemma} \label{lemmaaux}
If Assumption \ref{mainassumption2} holds  and  sequence $x^k$ is
generated by Algorithm (2-RCD) using  the uniform distribution, then
the following inequality is valid:
\begin{align*}
 & E[\psi_{L_{i_kj_k}\textbf{1}}(d_{i_kj_k};x^k)| \xi^{k-1} ] \\
 & \quad \le \left(1 - \frac{2}{N(N-1)}\right) F(x^k) + \frac{2}{N(N-1)}
 \psi_{N\Gamma}(d_{N\Gamma}(x^k);x^k) \quad  \forall k \ge 0.
\end{align*}
\end{lemma}
\begin{proof}:
As in the previous sections, for a simple exposition we drop $k$
from our derivations: e.g. the current point is denoted $x$, next
iterate $x^+=x+ U_id_i+U_jd_j$, where direction $d_{ij}$ is given by
Algorithm (2-RCD) for some random selection of pair $(i,j)$ and
$\xi$ instead of $\xi^{k-1}$. From the relation \eqref{psialpha} and
the property of minimizer $d_{ij}$ we have:
\[  \psi_{L_{ij}\textbf{1}}(d_{ij};x) \leq \psi_{L_{ij}\textbf{1}}(s_{ij};x) \quad
\forall   s_{ij} \in S_{ij}. \] Taking expectation in both sides
w.r.t. random variable $(i, j)$ conditioned on $\xi$ and recalling
that $p_{ij} = \frac{2} {N(N-1)}$, we get:
\begin{align*}
& E[\psi_{L_{ij}\textbf{1}}(d_{ij};x)| \; \xi ] \\
 & \le  f(x)+\frac{2}{N(N-1)} \Big[\sum\limits_{i,j} \langle
 \nabla_{ij}f(x),s_{ij} \rangle + \sum\limits_{i,j}\frac{L_{ij}}{2}\norm{s_{ij}}^2 +
  \sum\limits_{i,j}h(x+s_{ij}) \Big] \\
& =  f(x)+\frac{2}{N(N-1)} \Big[ \sum\limits_{i,j} \langle
 \nabla_{ij}f(x),s_{ij} \rangle + \sum\limits_{i,j}\frac{1}{2}\norm{\sqrt{L_{ij}}s_{ij}}^2 + \sum\limits_{i,j}h(x+s_{ij}) \Big],
\end{align*}
for all $s_{ij} \in S_{ij}$.  We can apply Lemma \ref{lemma2} for
coordinatewise separable functions $\norm{\cdot}^2$ and $h(\cdot)$
and we obtain:
\begin{align*}
E[\psi_{L_{ij}\textbf{1}}(d_{ij};x)| \; \xi ] \le
&f(x)+\frac{2}{N(N-1)}
  \Big[ \langle \nabla f(x), \sum\limits_{i,j} s_{ij} \rangle +
   \frac{1}{2}\norm{\sum\limits_{i,j} \sqrt{L_{ij}}s_{ij}}^2\\
 & \; +h(x+\sum\limits_{i,j} s_{ij})+
 \left(\frac{N(N-1)}{2}\!-\!1\right) h(x), \Big]
\end{align*}
for all $s_{ij} \in S_{ij}$. From Lemma \ref{lemma1} it follows that
any $s \in S$ has a conformal realization defined by $ s = \sum_t
s^t$, where the vectors $s^t \in S$ are elementary vectors conformal
to $s$. Therefore, observing that every elementary vector $s^t$ has
at most two nonzero blocks, then any vector $s \in S$ can be
generated by $s = \sum_{i,j} s_{ij}$, where $s_{ij} \in S$ are
conformal to $s$ and have at most two nonzero blocks, i.e. $s_{ij}
\in S_{ij}$ for some pair $(i,j)$. Due to conformal property of the
vectors $s_{ij}$, the expression $\norm{\sum_{i,j} \sqrt{L_{ij}}
s_{ij}}^2$ is nondecreasing in the weights $L_{ij}$ and taking in
account that $L_{ij}\le \min\{N\Gamma_i,N\Gamma_j\}$, the previous inequality
leads to:
\begin{align*}
& E[\psi_{L_{ij}\textbf{1}}(d_{ij};x)| \; \xi ] \\
&\le f(x) + \frac{2}{N(N-1)} \Big[ \langle \nabla f(x),
\sum\limits_{i,j} s_{ij} \rangle +
\frac{1}{2}\norm{\sum\limits_{i,j}D_{N\Gamma}^{1/2}s_{ij}}^2 +
 h(x+\sum\limits_{i,j} s_{ij}) \\
 & \qquad  + \left(\frac{N(N-1)}{2}-1\right) h(x) \Big] \\
&\!\!=\! f(x) \!+\! \frac{2}{N(N \!-\! 1)} \Big[ \langle \nabla
f(x), s \rangle \!+\! \frac{1}{2}\norm{\sqrt{N}D_\Gamma^{1/2} s}^2 \!+\!
h(x \!+\! s) \!+\! \Big(\frac{N(N \!-\!1)}{2} \!-\!1 \Big) h(x),
\Big]
\end{align*}
for all $s \in S$.  As the last inequality holds for any vector $s
\in S$, it also holds for the particular vector $d_{N\Gamma}(x) \in S$:
\begin{align*}
E[\psi_{L_{ij}\textbf{1}}(d_{ij};x)| \xi ]& \le\left(1 - \frac{2}{N(N-1)}\right) F(x) + \frac{2}{N(N-1)}\Big[ f(x) + \\
&\qquad \langle \nabla f(x), d_{N\Gamma}(x)\rangle +\frac{N}{2}\norm{d_{N\Gamma}(x)}_{\Gamma}^2\!+\!h(x\!+\!d_{N\Gamma}(x))\Big]\\
&=\left(1 - \frac{2}{N(N-1)}\right) F(x) + \frac{2}{N(N-1)}\psi_{N\Gamma}(d_{N\Gamma}(x);x).
\end{align*}
\qed
\end{proof}


\noindent The main convergence properties of Algorithm (2-RCD) are
given in the following theorem:
\begin{theorem}
\label{lemma3} If Assumption \ref{mainassumption2} holds for the
composite objective function F of problem \eqref{modelcon} and the
sequence $x^k$ is generated by Algorithm (2-RCD) using the uniform
distribution, then the following statements are valid:
\begin{enumerate}
\item[(i)] The sequence of random variables $M_2(x^k,\Gamma)$ converges  to 0 a.s.
and the sequence $F(x^k)$ converges to a random variable $\bar F$
a.s.
\item[(ii)] Any accumulation point of the sequence $x^k$ is a stationary
point for optimization problem ~\eqref{modelcon}.
\end{enumerate}
\end{theorem}

\begin{proof}:
(i) Using a similar reasoning as in Lemma \ref{lemadec} but for the
inequality \eqref{objquadij} we can show the following decrease in
the objective function for Algorithm (2-RCD) (i.e.  Algorithm
(2-RCD) is also a descent method):
\begin{equation}\label{lemadec2}
F(x^{k+1}) \le F(x^k) - \frac{L_{i_kj_k}}{2}\norm{d_{i_kj_k}}^2
\quad \forall k \geq 0.
\end{equation}
Further, subtracting $F^*$ from both sides, applying expectation
conditioned on $\xi^{k-1}$ and then using supermartingale
convergence theorem given in Lemma \ref{mart} we obtain  that
$F(x^k)$ converges to a random variable $\bar F$ a.s. for $k \to
\infty$. Due to almost sure convergence of sequence $F(x^k)$, it can
be easily seen that $\lim\limits_{k \to \infty} F(x^k) -
F(x^{k+1})=0$ a.s. Moreover, from \eqref{lemadec2} we have:
\begin{equation*}
 \frac{L_{i_kj_k}}{2}\norm{d_{i_k j_k}}^2 = \frac{L_{i_kj_k}}{2}\norm{x^{k+1}-x^k}^2 \le F(x^k) - F(x^{k+1})
\quad \forall k \geq 0,
\end{equation*}
which implies that
$$\lim\limits_{k \to \infty} d_{i_kj_k} = 0 \quad \text{and} \quad
  \lim\limits_{k \to \infty} \norm{x^{k+1}-x^k}=0 \quad \text{a.s.}$$

\noindent As in the previous section, for a simple exposition we
drop $k$ from our derivations: e.g. the current point is denoted
$x$, next iterate  $x^+=x+ U_id_i+U_jd_j$, where direction $d_{ij}$
is given by Algorithm (2-RCD) for some random selection of pair
$(i,j)$ and $\xi$ stands for $\xi^{k-1}$. From Lemma \ref{lemmaaux},
we obtain a sequence which bounds from below
$\psi_{N\Gamma}(d_{N\Gamma}(x);x)$ as follows:
\begin{equation*}
\frac{N(N-1)}{2} E[\psi_{L_{ij}\textbf{1}}(d_{ij};x)| \;
\xi]+\left(1-\frac{N(N-1)}{2} \right) F(x) \le
\psi_{N\Gamma}(d_{N\Gamma}(x);x).
\end{equation*}
On the other hand, from Lemma \ref{lemma1} it follows that any $s
\in S$ has a conformal realization defined by $s = \sum_{i,j}
s_{ij}$, where $s_{ij} \in S$ are conformal to $s$ and have at most
two nonzero blocks, i.e. $s_{ij} \in S_{ij}$ for some pair $(i,j)$.
Using now Jensen inequality we derive another sequence which bounds
$\psi_{N\Gamma}(d_{N\Gamma}(x);x)$ from above:
\begin{align*}
& \psi_{N\Gamma}(d_{N\Gamma}(x);x))= \min\limits_{s \in S} f(x)+ \langle
\nabla f(x), s\rangle + \frac{1}{2}\norm{s}_{N\Gamma}^2 + h(x+s)\\
& =\min\limits_{s_{ij} \in S_{ij}} \Big [f(x)+\langle \nabla f(x),
\sum\limits_{i,j} s_{ij}\rangle + \frac{1}{2}\norm{\sum\limits_{i,j} s_{ij}}_{N\Gamma}^2 +h(x+\sum\limits_{i,j} s_{ij}) \Big ]\\
& =\min\limits_{\tilde{s}_{ij} \in S_{ij}} f(x)+\frac{1}{N(N-1)}\langle \nabla f(x), \sum\limits_{i,j}
\tilde{s}_{ij}\rangle + \frac{1}{2}\norm{\frac{1}{N(N-1)}\sum\limits_{i,j}\tilde{s}_{ij}}_{N\Gamma}^2 \\
& \qquad +h\left(x+\frac{1}{N(N-1)}\sum\limits_{i,j}\tilde{s}_{ij}\right)\\
& \le \min\limits_{\tilde{s}_{ij} \in S_{ij}}
f(x)+\frac{1}{N(N-1)}\sum\limits_{i,j}\langle \nabla
f(x),\tilde{s}_{ij}\rangle
+ \frac{1}{2N(N-1)}\sum\limits_{i,j}\norm{\tilde{s}_{ij}}_{N\Gamma}^2 \\
& \qquad
+\frac{1}{N(N-1)}\sum\limits_{i,j}h\left(x+\tilde{s}_{ij}\right)=
E[\psi_{N\Gamma}(d_{ij};x)| \xi],
\end{align*}
where we used the notation $\tilde{s}_{ij}=N(N-1)s_{ij}$. If we come
back to the notation dependent on  $k$, then using Assumption
\ref{mainassumption2} (ii) and the fact that $d_{i_kj_k} \to 0$ a.s.
we obtain that  $E[\psi_{N\Gamma}(d_{i_kj_k};x^k)| \xi^{k-1}]$ converges
to $\bar F$ a.s. for  $k \to \infty$. We conclude that both
sequences, lower and upper bounds of $\psi_{N\Gamma}(d_{N\Gamma}(x^k);x^k)$
from above, converge to $\bar F$ a.s., hence
$\psi_{N\Gamma}(d_{N\Gamma}(x^k);x^k)$ converges to $\bar F$ a.s. for $k \to \infty$.
A trivial case of strong convexity relation
\eqref{strongconvex} leads to:
\begin{equation*}
 \psi_{N\Gamma}(0;x^k) \ge \psi_{N\Gamma}(d_{N\Gamma}(x^k);x^k) + \frac{N}{2} \norm{d_{N\Gamma}(x^k)}_\Gamma^2.
\end{equation*}
Note that  $\psi_{N\Gamma}(0;x^k)=F(x^k)$ and since both sequences
$\psi_{N\Gamma}(0;x^k)$ and  $\psi_{N\Gamma}(d_{N\Gamma}(x^k);x^k)$  converge to
$\bar F$ a.s. for $k \to \infty$, from  the above strong convexity
relation it follows that the sequence $M_2(x^k;\Gamma)=
\norm{d_{N\Gamma}(x^k)}_{\Gamma}$ converges to $0$ a.s. for  $k \to \infty$.

\noindent (ii) The proof follows the same ideas as in the proof of
Theorem \ref{lemmaconv} (ii). \qed
\end{proof}


We  now present the convergence rate   for Algorithm (2-RCD).
\begin{theorem}\label{constrainedrate}
Let $F$ satisfy Assumption \ref{mainassumption2}. Then, the
Algorithm (2-RCD) based on the uniform distribution generates a
sequence $x^k$ satisfying the following convergence rate for the
expected values of the optimality measure:
\begin{equation*}
\min\limits_{0 \le l \le k}E \left[ \left( M_2(x^l,\Gamma) \right)^2
\right] \le \frac{N \left(F(x^0) - F^* \right)}{k+1}  \quad \forall
k \geq 0.
\end{equation*}
\end{theorem}

\begin{proof}:
Given the current feasible point $x$, denote $x^+=x+ U_id_i+U_jd_j$
as the  next iterate, where direction $( d_i, d_j) $ is given by
Algorithm (2-RCD) for some random chosen pair $(i,j)$ and   we use
the notation  $(\phi, \phi^+, \xi)$ instead of $(\phi^k, \phi^{k+1},
\xi^{k-1})$. Based on Lipschitz inequality \eqref{objquadij} we
derive:
\begin{align*}
F(x^+) \le f(x) + \langle \nabla_{ij}f(x), d_{ij} \rangle +
 \frac{L_{ij}}{2}\norm{d_{ij}}^2 + h(x+ d_{ij}).
\end{align*}
Taking expectation conditioned on $\xi$ in both sides and using
Lemma \ref{lemmaaux} we get:
\begin{align*}
 E&[F(x^+)| \xi] \le \left(1 - \frac{2}{N(N-1)}\right) F(x) +
\frac{2}{N(N-1)}\psi_{N\Gamma}(d_{N\Gamma}(x);x).
\end{align*}
Taking now expectation w.r.t. $\xi$, we can derive:
\begin{align*}
& \phi -  \phi^+ \\
&\ge E[\psi_{N\Gamma}(0;x)] \!-\! \Big(\! 1 \!-\!
\frac{2}{N(N \!-\! 1)} \Big) \! E[\psi_{N\Gamma}(0;x)] \!-\! \frac{2}{N(N \!-\! 1)}E[\psi_{N\Gamma}(d_{N\Gamma}(x);x)]\\
&= \frac{2}{N(N-1)} \left( E[\psi_{N\Gamma}(0;x)] - E[\psi_{N\Gamma}(d_{N\Gamma}(x);x)] \right)\\
&\ge \frac{1}{N-1} E \left[ \norm{d_{N\Gamma}(x)}_{\Gamma}^2 \right] \ge
\frac{1}{N} E \left[ \left( M_2(x,\Gamma) \right)^2 \right],
\end{align*}
where we used the strong convexity property of function
$\psi_{N\Gamma}(s;x)$. Now, considering  iteration $k$ and summing up
with respect to  entire history we get:
\begin{equation*}
\frac{1}{N}  \sum\limits_{l=0}^k  E \left[ \left( M_2(x^l,\Gamma)
\right)^2 \right] \le F(x^0) - F^*.
\end{equation*}
This inequality leads us to the above result. \qed
\end{proof}


\subsection{Constrained minimization of smooth objective functions}

\noindent We now study the convergence of Algorithm (2-RCD) on the
particular case of optimization model \eqref{modelcon} with $h = 0$.
For this particular case a feasible point $x^*$   is a
\emph{stationary point} for \eqref{modelcon} if there exists
$\lambda^* \in \rset^{}$ such that:
\begin{equation}\label{optimcouple}
 \nabla f(x^*) + \lambda^* a =0  \quad \text{and}  \quad  a^Tx^*=b.
\end{equation}
For any feasible point $x$, note that exists $\lambda \in \rset^{}$
such that:
\begin{equation*}
 \nabla f(x) = \nabla f(x)_{\perp} - \lambda a,
\end{equation*}
where $\nabla f(x)_{\perp}$ is the projection of the gradient vector
$ \nabla f(x)$ onto the subspace $S$ orthogonal to the vector $a$.
Since $\nabla f(x)_{\perp} = \nabla f(x) + \lambda a$,   we defined
a particular optimality measure: \[ M_3(x, \textbf{1}) =
\norm{\nabla f(x)_{\perp}}. \] In this case the iteration of
Algorithm (2-RCD) is a projection onto a hyperplane so that the
direction $ d_{i_kj_k}$ can be computed in closed form. We denote by
$Q_{ij} \in \rset^{n \times n}$ the symmetric matrix with all blocks
zeros  except:
\begin{equation*}
Q_{ij}^{ii} = I_{n_i} - \frac{a_i a_i^T}{a_i^T a_i}, \ \
Q_{ij}^{ij}= - \frac{a_ia_j^T}{a_{ij}^Ta_{ij}}, \ \ Q_{ij}^{jj}=
I_{n_j} - \frac{a_ja_j^T}{a_{ij}^Ta_{ij}}.
\end{equation*}
It is straightforward to see that $Q_{ij}$ is positive semidefinite
(notation $Q_{ij} \succeq 0$) and $Q_{ij} a = 0$ for all pairs
$(i,j)$ with $i \neq j$. Given a probability distribution $p_{ij}$,
let us define the matrix:
\begin{equation*}
 Q = \sum\limits_{i, j} \frac{p_{ij}}{L_{ij}} Q_{ij},
\end{equation*}
that is also symmetric and positive semidefinite, since $L_{ij},
p_{ij} > 0$ for all $(i, j)$. Furthermore, since we consider all
possible pairs $(i,j)$, with $i \not = j \in \{1, \dots, N\}$, it
can be shown that the  matrix $Q$   has an eigenvalue $\nu_1(Q) = 0$
(which is a simple eigenvalue) with the associated eigenvector $a$.
It follows that $\nu_2(Q)$ (the second smallest eigenvalue of $Q$)
is positive. Since $h=0$, we have $F=f$. Using the same reasoning as
in the previous sections we can easily show that the sequence
$f(x^k)$ satisfies the following decrease:
\begin{equation}\label{coupledecrease}
f(x^{k+1}) \le f(x^k) - \frac{1}{2 L_{ij}} \nabla f(x^k)^T
Q_{ij}\nabla f(x^k) \quad \forall k \geq 0.
\end{equation}
We now give the convergence rate of Algorithm (2-RCD) for this
particular case:

\begin{theorem}
\label{constrainedrateh} Let $h = 0$ and $f$ satisfy Assumption
\ref{mainassumption2} (i). Then,  Algorithm (2-RCD) based on a
general probability distribution $p_{ij}$ generates a sequence $x^k$
satisfying the following convergence rate for the expected values of
the norm of the projected gradients onto subspace $S$:
\begin{equation*}
\min\limits_{0 \le l \le k} E \left[ \left(M_3(x^l, \textbf{1})
\right)^2 \right] \le \frac{2(F(x^0) - F^*)}{ \nu_2(Q)(k+1)}.
\end{equation*}
\end{theorem}
\begin{proof}
As in the previous section, for a simple exposition we drop $k$ from
our derivations: e.g. the current point is denoted $x$, and $x^+=x+
U_id_i+U_jd_j$, where direction $d_{ij}$ is given by Algorithm
(2-RCD) for some random selection of pair $(i,j)$. Since $h=0$, we
have $F=f$. From \eqref{coupledecrease} we have the following
decrease: $f(x^+) \le f(x) - \frac{1}{2 L_{ij}} \nabla f(x)^T
Q_{ij}\nabla f(x)$. Taking now expectation conditioned in $\xi$ in
this inequality  we have:
\begin{equation*}
E[f(x^+) | \; \xi] \le f(x) - \frac{1}{2} \nabla f(x)^T Q \nabla
f(x).
\end{equation*}
From the above decomposition of  the gradient  $\nabla f(x)= \nabla
f(x)_{\perp}  -  \lambda a$ and the observation that $Qa=0$, we
conclude that the previous inequality does not change if we replace
$\nabla f(x)$ with $\nabla f(x)_{\perp}$:
\begin{equation*}
E[f(x^+) | \xi] \le f(x) - \frac{1}{2} \nabla f(x)_{\perp}^T Q
\nabla f(x)_{\perp}.
\end{equation*}
Note that $\nabla f(x)_{\perp}$ is included in the orthogonal
complement of the span of vector $a$, so that the above inequality
can be relaxed to:
\begin{equation}\label{relaxeddecrease}
E[f(x^+) | \; \xi] \le f(x) - \frac{1}{2} \nu_2(Q) \norm{\nabla
f(x)_{\perp}}^2 = f(x) - \frac{\nu_2(Q)}{2} \left( M_3(x,
\textbf{1}) \right)^2.
\end{equation}
Coming back to the notation dependent on $k$ and   taking
expectation in both sides of inequality \eqref{relaxeddecrease}
w.r.t. $\xi^{k-1}$,  we have:
\begin{equation*}
\phi^k - \phi^{k+1} \ge \frac{\nu_2(Q)}{2}E \left[ \left( M_3(x^k,
\textbf{1}) \right)^2 \right].
\end{equation*}
Summing w.r.t. the entire history, we obtain the above result. \qed
\end{proof}

Note  that our convergence proofs given in this section (Theorems 4,
5 and 6) are different from the convex case
\cite{Nec:13,NecPat:12,NecNes:12}, since here we introduce another
optimality measure  and we use supermartingale convergence theorem
in the analysis.  It is important to see  that the convergence rates
for the Algorithm (2-RCD) given in Theorems \ref{constrainedrate}
and \ref{constrainedrateh} are typical for the class of first order
methods designed for solving nonconvex and nonsmotth optimization
problems, e.g. in \cite{Bec:12,Nes:07} similar results are obtained
for other gradient based methods designed to solve nonconvex
problems.


\section{Numerical Experiments}
In this  section we analyze the practical performance  of the random
coordinate descent methods derived in this paper and compare our
algorithms with some recently developed state-of-the-art algorithms
from the literature. Coordinate descent methods are one of the most
efficient classes of algorithms for large-scale optimization
problems. Therefore,  we   present extensive numerical simulation
for large-scale nonconvex problems with dimension ranging from
$n=10^3$ to $n=10^7$. For numerical experiments,  we  implemented
all the algorithms in C code and we  performed our tests on a PC
with Intel Xeon E5410 CPU and 8 Gb RAM memory.

\noindent For tests we choose as application the \textit{eigenvalue
complementarity problem}.  It is well-known that many applications
from mathematics, physics and engineering require the efficient
computation of eigenstructure of some symmetric matrix. A brief list
of these applications includes optimal control, stability analysis
of dynamic systems, structural dynamics, electrical networks,
quantum chemistry, chemical reactions and economics (see
\cite{FaiMar:12,MonTor:04,Par:97,ThiMoe:10} and the reference
therein for more details). The eigenvalues of a symmetric matrix $A$
have an elementary definition as the roots of the characteristic
polynomial $det(A-\lambda I)$. In realistic applications the
eigenvalues can have an important role, for example to describe
expected long-time behavior of a dynamical system, or to be only
intermediate values of a computational method. For many applications
the optimization approach for eigenvalues computation is better than
the algebraic one. Although, the eigenvalues computation can be
formulated as a convex problem, the corresponding feasible set is
complex so that the projection on this set is numerically very
expensive, at least of order  $\mathcal{O}(n^2)$. Therefore,
classical methods for convex optimization are not adequate for
large-scale eigenvalue problems. To obtain a lower iteration
complexity as $\mathcal{O}(n)$ or even $\mathcal{O}(p)$, where $p
\ll n$, an appropriate way to approach these problems is through
nonconvex formulation and using coordinate descent methods. A
classical optimization problem formulation involves the Rayleigh
quotient as the objective function of some nonconvex optimization
problem \cite{MonTor:04}. The eigenvalue complementarity problem
(EiCP) is an extension of the classical eigenvalue problem, which
can be stated as: {\it given matrices A and B, find $\nu \in
\rset^{}$ and $x \neq 0$ such that}
$$ \begin{cases}
      w = (\nu B - A)x, \\
      w \ge 0,\; x \geq 0, \; w^Tx=0.
 \end{cases}
$$
If matrices A and B are symmetric, then we have \textit{symmetric
(EiCP)}. It has been shown in \cite{ThiMoe:10} that symmetric (EiCP)
is equivalent with finding a stationary point of a generalized
Rayleigh quotient on the simplex:
\begin{align*}
 \min\limits_{x \in \rset^n} \ & \frac{x^T A x}{x^T B x} \\
   \text{s.t.:} \ \ & \textbf{1}^T x = 1, \;  x \ge 0,
\end{align*}
where we recall that  $\textbf{1}=[1 \dots 1]^T \in \rset^n$. A
widely used alternative formulation of (EiCP) problem is the
nonconvex logarithmic formulation (see \cite{JudRay:08,ThiMoe:10}):
\begin{align}
   \max\limits_{x \in \rset^n}& \  f(x) \;\; \left(= \ln{\frac{x^TAx}{x^TBx}} \right) \label{logform}\\
   \text{s.t.:} \ \ &  \textbf{1}^T x = 1, \; x \ge 0. \nonumber
\end{align}
Note that optimization problem \eqref{logform} is a particular case
of \eqref{modelcon}, where $h$ is the indicator function of  the
nonnegative orthant. In order to have a well-defined objective
function for the logarithmic case, in the most of the aforementioned
papers the authors assumed positive definiteness of matrices $A
=[a_{ij}]$ and $B = [b_{ij}]$. In this paper, in order to have a
more practical application with a highly nonconvex objective
function \cite{FaiMar:12}, we consider the class of nonnegative
matrices, i.e. $A, B \ge 0$, with positive diagonal elements, i.e.
$a_{ii} > 0 $ and $b_{ii} > 0$ for all $i=1, \cdots, n$. For this
class of matrices the problem \eqref{logform} is also well-defined
on the simplex. Based on Perron-Frobenius theorem, we have that for
matrices $A$ that are also irreducible and $B=I_n$ the corresponding
stationary point of the (EiCP) problem \eqref{logform} is the global
minimum of this problem or equivalently is the Perron vector, so
that any accumulation point of the sequence generated  by our
Algorithm (2-RCD) is also a global minimizer. In order to apply our
Algorithm (2-RCD) on the logarithmic formulation of the (EiCP)
problem \eqref{logform}, we have to compute an approximation of the
Lipschitz constants $L_{ij}$.  For brevity, we introduce the
notation $\Delta_n = \{x \in \rset^n : \textbf{1}^T x = 1, \;  x \ge
0 \}$ for the standard simplex  and the function $g_A(x) =
\ln{x^TAx}$. For a given matrix $A$, we denote by $A_{ij} \in
\rset^{(n_i+n_j) \times (n_i+n_j)}$ the $2 \times 2$ block matrix of
$A$ by taking the pair $(i,j)$ of block rows of matrix $A$ and then
the pair $(i,j)$ of block columns of $A$.
\begin{lemma}
Given a nonnegative matrix $A \in \rset^{n \times n}$ such that
$a_{ii} \neq 0$ for all $i=1, \cdots, n$, then the function $g_A(x)
=  \ln{x^TAx}$ has 2 block coordinate Lipschitz gradient on the
standard simplex, i.e.:
\begin{equation*}
 \norm{\nabla_{ij}g_A(x+s_{ij}) - \nabla_{ij}g_A(x)} \le
 L_{ij}^A\norm{s_{ij}}, \;\; \forall x, x+s_{ij} \in \Delta_n,
\end{equation*}
where an upper bound on  Lipschitz constant $L_{ij}^A$ is given by
\begin{equation*}
L_{ij}^A \le \frac{2N}{\min\limits_{1 \le i \le N} a_{ii}}
\norm{A_{ij}}.
\end{equation*}
\end{lemma}
\begin{proof}:
The Hessian of the function $g_A(x)$ is given by \[ \nabla^2 g_A(x)=
\frac{2A}{x^TAx} - \frac{4(Ax)(Ax)^T}{(x^TAx)^2}. \] Note that
$\nabla^2_{ij} g_A(x) = \frac{2A_{ij}}{x^TAx} -
\frac{4(Ax)_{ij}(Ax)_{ij}^T}{(x^TAx)^2}$. With the same arguments as
in \cite{ThiMoe:10} we have that: $\norm{\nabla^2_{ij} g_A(x)} \le
\norm{\frac{2A_{ij}}{x^TAx}}$. From the mean value theorem we
obtain:
\begin{equation*}
 \nabla_{ij}g_A(x+s_{ij}) = \nabla_{ij}g_A(x)+ \int_0^1
  \nabla_{ij}^2 g_A(x+\tau s_{ij}) \ s_{ij} \ \mathrm{d}\tau,
\end{equation*}
for any $x,x+s_{ij} \in \Delta_n$. Taking norm in both sides of the
equality results in:
\begin{align*}
&  \norm{\nabla_{ij}g_A(x+s_{ij}) - \nabla_{ij}g_A(x)} =
 \norm{\left(\int_0^1 \nabla_{ij}^2 g_A(x+\tau s_{ij}) \ \mathrm{d}\tau\right) s_{ij}} \\
&\le \int_0^1 \norm{\nabla_{ij}^2 g_A(x+\tau s_{ij})} \
\mathrm{d}\tau \ \norm{s_{ij}} \le \norm{\frac{2A_{ij}}{x^TAx}} \
\norm{s_{ij}} \ \ \forall x,x+s_{ij} \in \Delta_n.
\end{align*}
Note that $\min\limits_{x \in \Delta_n}x^T A x >0$ since we have:
\begin{equation*}
 \min\limits_{x \in \Delta_n} x^TAx \ge \min\limits_{x \in \Delta_n}
 \left(\min\limits_{1\le i\le n} a_{ii} \right)\norm{x}^2 =\frac{1}{N}\min\limits_{1\le i\le n}
 a_{ii}.
\end{equation*}
and the above result can be easily derived. \qed
\end{proof}
\noindent Based on the previous  notation, the objective function of
the logarithmic formulation \eqref{logform} is given by:
\[ \max\limits_{x \in \Delta_n} f(x) \quad (= g_A(x) - g_B(x)) \quad \text{or} \quad
 \min\limits_{x \in \Delta_n} \bar f(x) \quad (= g_B(x) - g_A(x)). \]
Therefore,  the local Lipschitz constants $L_{ij}$ of function $f$
are estimated very easily and numerically cheap as: \[ L_{ij} \le
L_{ij}^A+L_{ij}^B=\frac{2N}{\min\limits_{1 \le i \le n} a_{ii}}
\norm{A_{ij}} + \frac{2N}{\min\limits_{1 \le i \le n} b_{ii}}
\norm{B_{ij}}  \quad \quad \forall i \not= j.
\]

\noindent In \cite{ThiMoe:10} the authors show that a variant of
difference of convex functions (DC) algorithm is very efficient for
solving the logarithmic formulation \eqref{logform}. We present
extensive numerical experiments for evaluating the performance of
our Algorithm (2-RCD) in comparison with the Algorithm (DC). For
completeness, we also present the Algorithm (DC)  for logarithmic
formulation of (EiCP) in the minimization form from
\cite{ThiMoe:10}: given $x_0 \in \rset^n$, for  $k \geq 0$ do
 \begin{equation*}\label{algorithm}
\begin{split}
 & \textbf{Algorithm (DC) \cite{ThiMoe:10}}\\
&1.\ \text{Set} \ y^k = \left( \mu I_n + \frac{2A}{\langle x^k, Ax^k\rangle} - \frac{2B}{\langle x^k, Bx^k\rangle}\right)x^k,\\
&2. \ \text{Solve the QP}: x^{k+1}= \arg \min_{x \in \rset^n}
\left\{ \frac{\mu}{2}\norm{x}^2 - \langle x, y_k \rangle :
\textbf{1}^T x = 1, x \ge 0 \right\},
 \end{split}
 \end{equation*}
where $\mu$ is a parameter chosen in a preliminary stage of the
algorithm such that the function $x \mapsto \frac{1}{2}\mu
\norm{x}^2 + \ln (x^TAx)$ is convex. In both algorithms we use the
following stopping criterion: $ |f(x^k) - f(x^{k+1})| \le \epsilon$,
where $\epsilon$ is some chosen accuracy.  Note that Algorithm (DC)
is based on full gradient information and  in the application (EiCP)
the most computations consists of matrix vector multiplication and a
projection onto simplex. When at least one matrix $A$ and $B$ is
dense, the computation of the sequence $y^k$ is involved, typically
$\mathcal{O}(n^2)$ operations. However, when these matrices are
sparse the computation can be reduced to  $\mathcal{O}(pn)$
operations, where $p$ is the average number of nonzeros in each row
of the matrix $A$ and $B$. Further, there are efficient algorithms
for computing the projection onto simplex, e.g. block pivotal
principal pivoting algorithm described in \cite{JudRay:08}, whose
arithmetic complexity is of order $\mathcal{O}(n)$. As it appears in
practice, the value of parameter $\mu$ is crucial in the rate of
convergence of Algorithm (DC). The authors in \cite{ThiMoe:10}
provide an approximation of $\mu$ that can be computed easily when
the matrix $A$ from \eqref{logform} is positive definite. However,
for general copositive matrices (as the case of nonnegative
irreducible matrices considered in this paper) one requires the
solution of certain NP-hard problem to obtain a good approximation
of parameter $\mu$. On the other hand, for our Algorithm (2-RCD) the
computation of the Lipschitz constants $L_{ij}$ is very simple and
numerically cheap (see previous lemma). Further, for the scalar case
(i.e. $n=N$) the complexity per iteration of our method applied to
(EiCP) problem is $\mathcal{O}(p)$ in the sparse case.

\setlength{\extrarowheight}{1pt}
\renewcommand{\tabcolsep}{4pt}
\begin{table}[ht]
\centering \caption{Performance of Algorithms (2-RCD) and (DC) on
randomly generated (EiCP) sparse problems with $p=10$ and  random
starting point $x^0$ for different problem dimensions $n$.} {\small
\begin{tabular}{|c|c|c|c|c|c|c|c|c|c|}
\hline
\multirow{2}{*}{$n$}  &\multicolumn{4}{c|}{\textbf{(DC)}} &
\multicolumn{3}{c|}{\textbf{(2-RCD)}} \\
\cline{2-8}
 &$\mu$ &CPU (sec)& iter & $F^*$ & CPU (sec)& full-iter & $F^*$\\
\hline \hline \multirow{4}{*}{$5\cdot 10^3$}& $0.01n$ &
0.0001&1&1.32 & \multirow{4}{*}{0.09} & \multirow{4}{*}{56} &
\multirow{4}{*}{105.20}\\ \cline{2-5} & $n$ & 0.001 & 2& 82.28 & &
&\\ \cline{2-5} & 2$n$ & 0.02 & 18 & 105.21 & & &\\ \cline{2-5}
& 50$n$ & 0.25 & 492 & 105.21 & & &\\
\hline \multirow{4}{*}{$2\cdot 10^4$} & $0.01n$ &  0.01&1&1.56
&\multirow{4}{*}{0.39} &\multirow{4}{*}{50}
&\multirow{4}{*}{73.74}\\ \cline{2-5} & $n$ & 0.01 & 2 & 59.99 & &
&\\ \cline{2-5} & 1.43$n$ & 0.59 & 230 & 73.75 & & &\\ \cline{2-5}
& 50$n$ & 0.85 & 324 & 73.75 & & &\\
\hline \multirow{4}{*}{$5 \cdot 10^4$} & $0.01n$ &  0.01& 1&1.41
&\multirow{4}{*}{1.75} &\multirow{4}{*}{53}
&\multirow{4}{*}{83.54}\\ \cline{2-5} & $n$ & 0.02 & 2& 67.03 & &&\\
\cline{2-5} & 1.43$n$ & 1.53 & 163 & 83.55 & &&\\ \cline{2-5}
& 50$n$ & 2.88 & 324 & 83.57 & &&\\
\hline \multirow{4}{*}{$7.5 \cdot 10^4$} & $0.01n$ & 0.01& 1&2.40
&\multirow{4}{*}{3.60} &\multirow{4}{*}{61}
&\multirow{4}{*}{126.04}\\ \cline{2-5} & $n$ & 0.03 & 2& 101.76 &
&&\\ \cline{2-5} &1.45$n$ & 6.99 & 480 & 126.05 &  &&\\ \cline{2-5}
&50$n$ & 4.72 & 324 & 126.05 &  &&\\
\hline \multirow{4}{*}{$10^5$} & $0.01n$ &  0.02& 1&0.83
&\multirow{4}{*}{4.79} &\multirow{4}{*}{53} &\multirow{4}{*}{52.21
}\\ \cline{2-5} & $n$ & 0.05 & 2&  41.87 & &&\\ \cline{2-5} &
1.43$n$ & 6.48 & 319 & 52.22 & &&\\ \cline{2-5}
& 50$n$ & 6.57  & 323 & 52.22 & &&\\
\hline \multirow{4}{*}{5 $\cdot10^5$} & $0.01n$ &  0.21&1&2.51
&\multirow{4}{*}{49.84} &\multirow{4}{*}{59}
&\multirow{4}{*}{136.37}\\ \cline{2-5} & $n$ & 0.42 & 2& 109.92
&&&\\ \cline{2-5} & 1.43$n$ & 94.34 & 475 & 136.38 &  &&\\
\cline{2-5}
& 50$n$ & 66.61 & 324 & 136.38 &  &&\\
\hline \multirow{4}{*}{7.5 $\cdot10^5$} & $0.01n$ &  0.44& 1&3.11
&\multirow{4}{*}{37.59} &\multirow{4}{*}{38}
&\multirow{4}{*}{177.52}\\ \cline{2-5} & $n$ & 0.81 & 2& 143.31 &
&&\\ \cline{2-5} & 1.43$n$ & 72.80 & 181 & 177.52 & &&\\ \cline{2-5}
& 50$n$ & 135.35 & 323 & 177.54 & &&\\
\hline \multirow{4}{*}{$10^6$} & $0.01n$ & 0.67& 1&3.60
&\multirow{4}{*}{49.67} &\multirow{4}{*}{42}
&\multirow{4}{*}{230.09}\\ \cline{2-5} &$n$ & 1.30& 2& 184.40 & &&\\
\cline{2-5} & 1.43$n$ & 196.38 & 293 & 230.09 &  &&\\ \cline{2-5}
& 50$n$ & 208.39 & 323 & 230.11 &  &&\\
\hline \multirow{4}{*}{$10^7$} & $0.01n$ & 4.69& 1&10.83
&\multirow{4}{*}{758.1} &\multirow{4}{*}{ 41}
&\multirow{4}{*}{272.37 }\\ \cline{2-5} & $n$ & 22.31 & 2 & 218.88
& &&\\ \cline{2-5} & 1.45$n$ & 2947.93 & 325 & 272.37& &&\\
\cline{2-5}
& 50$n$ &  2929.74& 323 &272.38 & &&\\
\hline
\end{tabular}
}
\end{table}

 In Table 1 we compare  the two algorithms: (2-CRD) and (DC). We
generated random sparse symmetric nonnegative and irreducible
matrices of dimension ranging from $n=10^3$ to $n=10^7$ using the
uniform distribution. Each row of the matrices  has  only $p=10$
nonzero entries.  In both algorithms we start from random initial
points. In the table we present for each algorithm the final
objective function value ($F^*$), the number of iterations (iter)
and the necessary CPU time (in seconds) for our computer to execute
all the iterations. As  Algorithm (DC) uses the whole gradient
information to obtain the next iterate, we also report for Algorithm
(2-RCD) the equivalent number of full-iterations which means the
total number of iterations divided by $n/2$ (i.e. the number of
iterations groups $x^0, x^{n/2}, . . . , x^{kn/2}$). Since computing
$\mu$ is very difficult for this type of matrices, we try to tune
$\mu$ in Algorithm (DC). We have tried four values for $\mu$ ranging
from $0.01n$ to $50n$. We have noticed that if $\mu$ is not
carefully tuned   Algorithm (DC) cannot find the optimal value $f^*$
in a reasonable time. Then, after extensive simulations we find an
appropriate value for $\mu$ such that  Algorithm (DC) produces an
accurate approximation of the optimal value. From the table we see
that  our Algorithm (2-RCD)  provides better performance in terms of
objective function values and CPU time (in seconds) than Algorithm
(DC). We also observe that our algorithm is not sensitive w.r.t. the
Lipschitz constants $L_{ij}$ and also w.r.t. the initial point,
while  Algorithm (DC) is very sensitive to the choice of $\mu$ and
the initial point.

\begin{figure}[ht]
\label{fig1} \centering \caption{Performance in terms of function
values of Algorithms (2-RCD) and (DC)  on a randomly generated
(EiCP) problem with $n= 5 \cdot 10^5$: left  $\mu=1.42 \cdot n$  and
right $\mu=50 \cdot n$. }
\includegraphics[width=5.8cm,height=4cm]{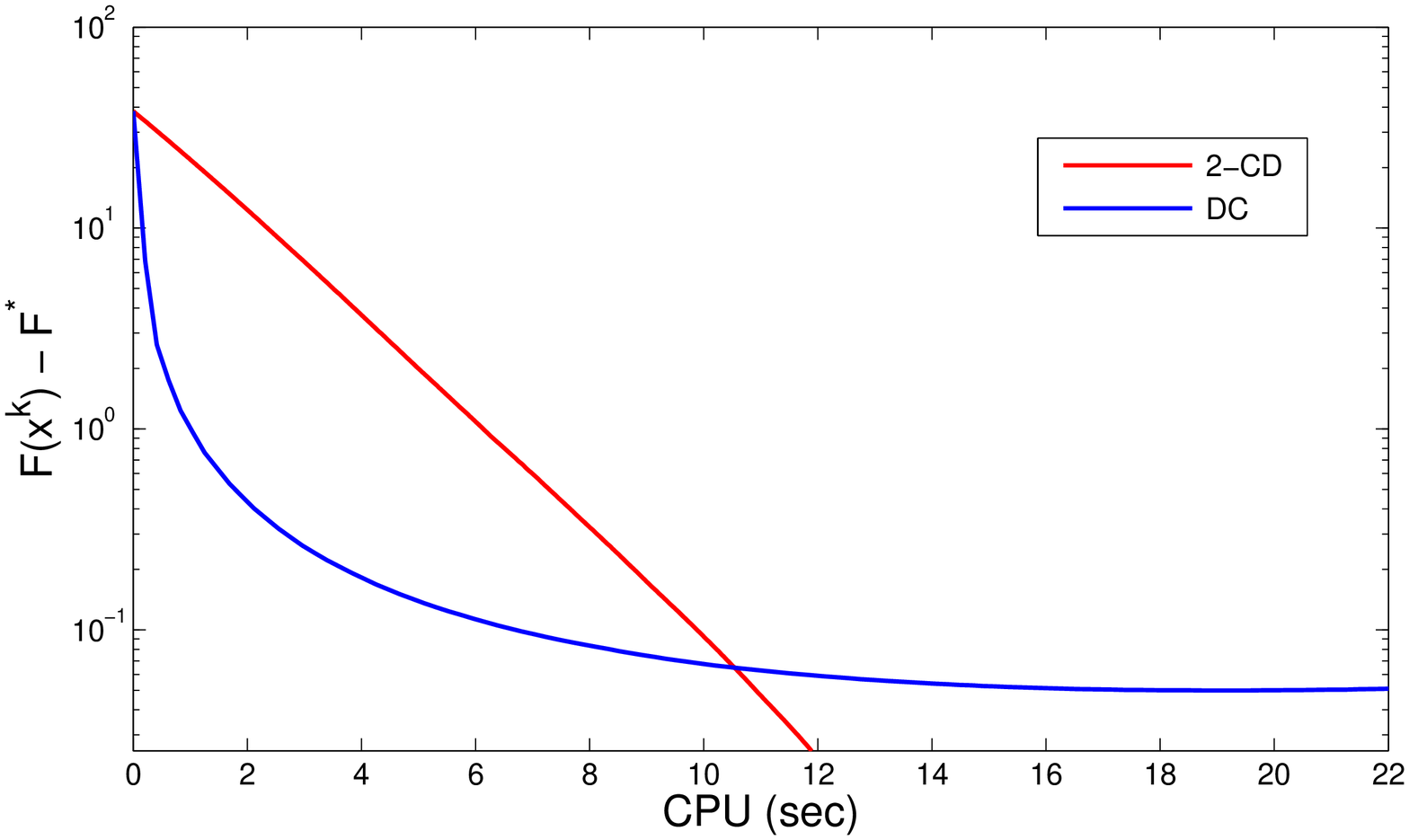}
\includegraphics[width=5.8cm,height=4cm]{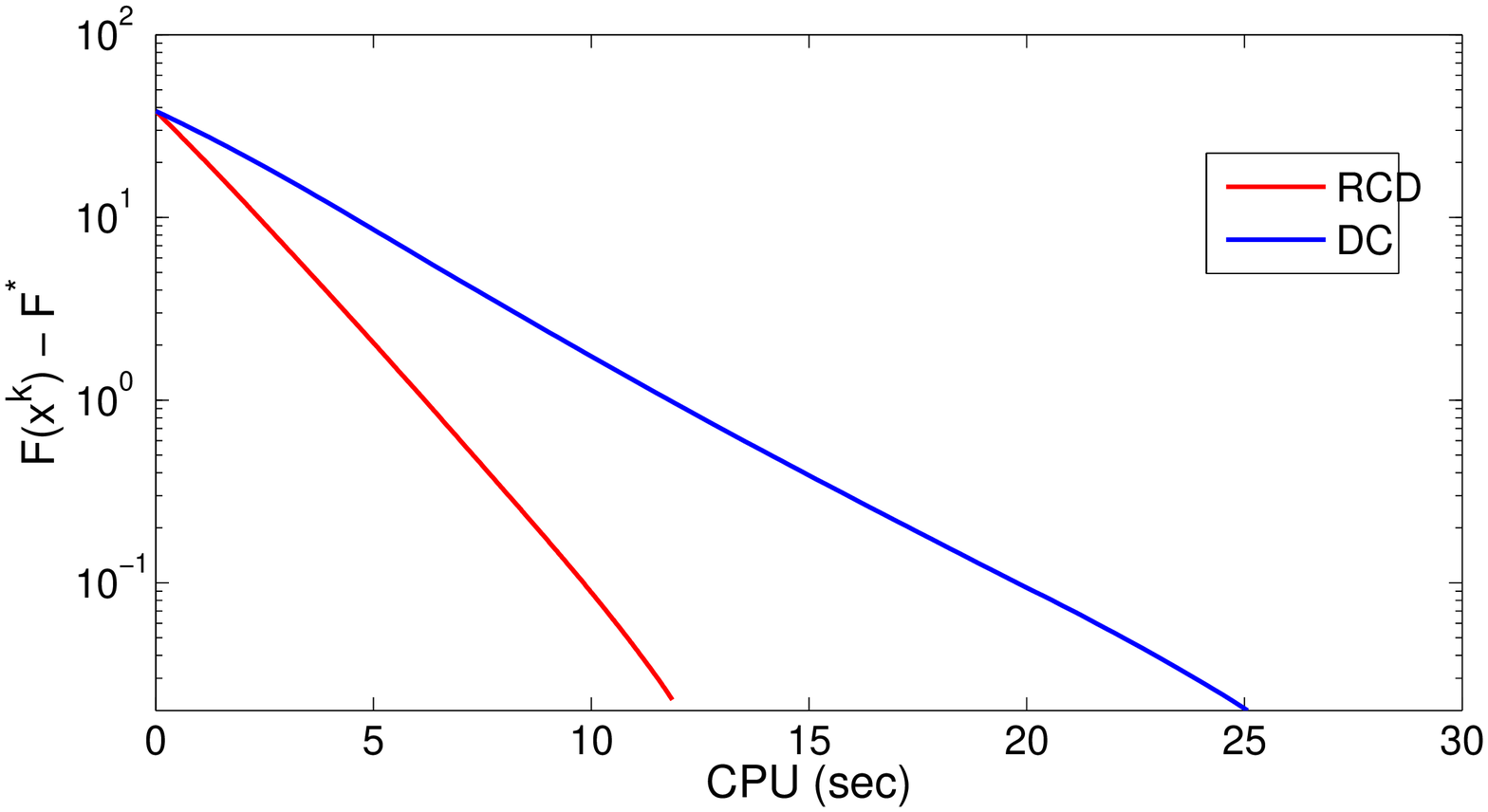}
\end{figure}
 Further, in Fig. 1 we plot the evolution of the objective function
w.r.t.  time for Algorithms (2-RCD) and (DC), in logarithmic scale,
on a random (EiCP) problem with dimension $n=5\cdot 10^5$ (Algorithm
(DC) with parameter left:  $\mu=1.42 \cdot n$; right: $\mu=50 \cdot
n$). For a good choice of  $\mu$ we see that in the initial phase of
Algorithm (DC) the reduction in the objective function is very fast,
but while approaching the optimum it slows down. On the other hand,
due to the sparsity and randomization our proposed algorithm is
faster in numerical implementation than the (DC) scheme.

\begin{figure}[ht]
\label{fig2} \centering \caption{CPU time performance of Algorithms
(2-RCD) and (DC)   for different values of the sparsity $p$ of the
matrix on a randomly generated (EiCP) problem of dimension $n = 2
\cdot 10^4$.}
\includegraphics[width=8cm,height=4cm]{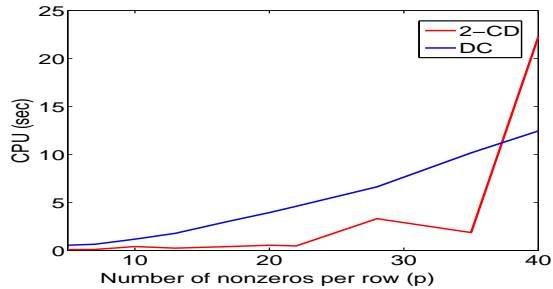}
\end{figure}
In Fig. 2 we plot the evolution of  CPU time, in logarithmic scale,
required for solving the problem  w.r.t. the average number of
nonzeros entries $p$ in each row of the matrix $A$. We see that for
very sparse matrices (i.e. for matrices with relatively small number
of nonzeros per row $p \ll n$), our Algorithm (2-RCD) performs
faster in terms of CPU time  than (DC) method. The main reason is
that our method has a simple implementation, does not require the
use of other algorithms at each iteration and the arithmetic
complexity of an iteration is of order $\mathcal{O}(p)$. On the
other hand, Algorithm (DC) is using the block pivotal principal
pivoting algorithm described in \cite{JudRay:08} at each iteration
for projection on simplex and the arithmetic complexity of an
iteration is of order $\mathcal{O}(pn)$.

We conclude from the theoretical  rate of convergence and the
previous numerical results  that Algorithms (1-RCD) and (2-RCD) are
easier to be implemented and analyzed due to the randomization and
the typically very simple iteration. Furthermore, on certain classes
of problems with sparsity structure, that appear frequently in many
large-scale real applications, the practical complexity  of our
methods is better than that of some well-known methods from the
literature. All these arguments make our algorithms  to be
competitive in the large-scale nonconvex optimization framework.
Moreover, our methods are suited for recently developed
computational architectures (e.g., distributed or parallel
architectures \cite{NecCli:13,RicTac:13}).


\end{document}